\documentclass[12pt]{amsart}
\usepackage[left=1in,top=1in,right=1in,bottom=1in]{geometry}
\usepackage{amssymb}
\usepackage{amsfonts}
\usepackage{amssymb,epsfig,amsfonts}

\newtheorem{theorem}{Theorem}

\newtheorem{lemma}[theorem]{Lemma}
\numberwithin{equation}{section}

\newtheorem{corollary}[theorem]{Corollary}

\begin{document}
\title[Distinguishing homomorphisms]{Distinguishing homomorphisms of
infinite graphs}
\author{Anthony Bonato}
\address{Department of Mathematics\\
Ryerson University\\
Toronto, ON\\
Canada, M5B 2K3} \email{abonato@ryerson.ca}
\author{Dejan Deli\'{c}}
\address{Department of Mathematics\\
Ryerson University\\
Toronto, ON\\
Canada, M5B 2K3} \email{ddelic@ryerson.ca}
\thanks{Supported by grants from NSERC, Mprime, and Ryerson}

\begin{abstract}
We supply an upper bound on the distinguishing chromatic number of certain infinite graphs satisfying
an adjacency property. Distinguishing proper $n$-colourings are generalized to the new notion of
distinguishing homomorphisms. We prove that if a graph $G$ satisfies the connected existentially closed property and admits a homomorphism to $H$, then it admits continuum-many distinguishing homomorphisms from $G$ to $H$ join $K_2.$ Applications
are given to a family universal $H$-colourable graphs, for $H$ a finite core.
\end{abstract}

\maketitle

\section{Introduction}

The distinguishing number is a widely studied graph parameter, first
introduced by Albertson and Collins \cite{ac}. Given a graph $G,$ its \emph{%
distinguishing number}, written $D(G),$ is the least positive integer $n$
such that there exists an $n$-colouring of $V(G)$ (not necessarily proper)
so that no non-trivial automorphism preserves the colours. The \emph{%
distinguishing chromatic number}, written $\chi _{D}$, is a variant of the
distinguishing number which requires that the $n$-colouring be proper (so the
set of vertices of a given colour forms an independent set). The
distinguishing chromatic number was introduced by Collins and Trenk~\cite{ct}
in 2006 (see also \cite{cht,ls}).

The distinguishing number of infinite graphs was first considered in \cite%
{imrich}. In particular, it was proved there that the distinguishing number
of the \emph{infinite random} (or \emph{Rado}) \emph{graph}, written $R,$ is
$2.$ (See \cite{cam} for background on $R.)$ This result was generalized first in
\cite{laflamme} and then in \cite{bd}; in the latter paper it was shown that graphs
satisfying a certain adjacency property have distinguishing number $2.$ As
the chromatic number of $R$ and many of its relatives (such as the Henson
universal homogeneous $K_{n}$-free graphs) are infinite, their
distinguishing chromatic numbers are also infinite. We find bounds on the
distinguishing chromatic numbers of certain infinite, symmetric graphs of bounded
chromatic number: the universal pseudo-homogeneous $H$-colourable graphs, where $H$ is
a finite core graph (see \cite{bonato0,mms}). This family
will be discussed in detail in Section~\ref{mainsec}.

We prove our results in the new and general setting of distinguishing homomorphisms (defined
in the next section). Distinguishing homomorphisms generalize distinguishing proper colourings,
and some of their properties are outlined in Lemma~\ref{lemma} in Section~2.
 Our main result is Theorem \ref{main}, which
demonstrates that for a graph $G$ satisfying a certain adjacency property (called c.e.c.)~which admits
a homomorphism to $H$,
there are continuum-many distinct distinguishing homomorphisms from $G$ to $H$ join $K_{2}.$ In
particular, for such graphs we derive the bound $\chi _{D}(G)\leq \chi (G)+2.$ We apply this
result to the universal pseudo-homogeneous $H$-colourable graphs.

Throughout, all graphs we consider are undirected, simple, and countable (that is, either finite or countably infinite). For background on graph theory, the reader is directed to \cite{diestel,west}. The cardinality of the continuum (that is, the set of real numbers) is denoted by $2^{\aleph_0}.$ For a function $f:X\rightarrow Y$ and $S\subseteq X$, we use the notation $f\upharpoonright S$ for the restriction of $f$ to $S.$ We use the notation $1_X$ for the identity function on $X.$ If $G$ is a graph, then its automorphism group is denoted $\mathrm{Aut}(G)$.

\section{Distinguishing homomorphisms}

The chromatic distinguishing number is defined in terms of proper $n$%
-colourings which are \emph{distinguishing}: no non-trivial automorphism
preserves the colours. A proper $n$-colouring may be viewed as a
homomorphism into $K_{n}$, which allows us to generalize this notion to
the setting of graph homomorphisms.

Fix a finite graph $H.$ For a graph $G$, a \emph{homomorphism} from $G$ to $H$ is a mapping $%
f:V(G)\rightarrow V(H)$ such that $xy\in E(G)$ implies that $f(x)f(y)\in
E(H).$ We abuse notation and write $f:G\rightarrow H,$ or even $G\rightarrow
H$ if the mention of $f$ is not important. We say that $G$ is $H$-\emph{colourable}. For additional
background on graph homomorphisms, see \cite{hell}.

A \emph{distinguishing homomorphism} from $G$ to $H\ $is a homomorphism $%
f:G\rightarrow H$ so that for all $\alpha \in \mathrm{Aut}(G),$ if%
\begin{equation}
\alpha f^{-1}=f^{-1}  \label{moo}
\end{equation}%
then $\alpha =1.$ We write $G\overset{D}{\rightarrow }$ $H$ if there is some
distinguishing homomorphism from $G\ $to $H.$ If $f:G\rightarrow H$ is any homomorphism and $\alpha $ satisfies (\ref%
{moo}), then we say it is \emph{preserving relative to} $f$. Note that if $%
\alpha $ is preserving, then for $x\in V(H)$ it permutes the elements of $%
f^{-1}(x)$  (we may think of each independent set $f^{-1}(x)$ as the
vertices all of one colour). Hence, a distinguishing proper $n$%
-colouring is just a distinguishing homomorphism to $K_{n}.$ For an example,
see Figure~\ref{fig1}. Note that an injective homomorphism is necessarily distinguishing (in particular, we usually consider only the case when $f^{-1}$ is a relation). Hence, every homomorphism from a core graph (that is, a
graph with the property that every homomorphism from $H$ to itself is an
automorphism) to itself is distinguishing.
\begin{figure} [h!]
\begin{center}
\epsfig{figure=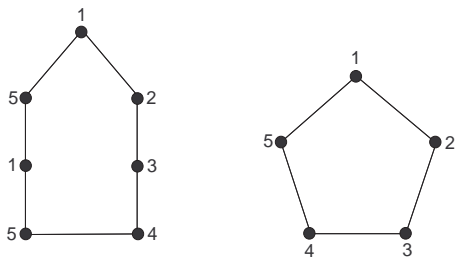,width=2.5in,height=1.5in}
\caption{A distinguishing homomorphism from $C_7$ to $C_5.$ The labels on $C_7$
describe the homomorphism.}\label{fig1}
\end{center}
\end{figure}

We prove the following lemma which collects some facts on distinguishing homomorphisms. A graph
$G$ is \emph{uniquely }$H$\emph{-colourable} if it is $H$-colourable, any
homomorphism from $G$ to $H$ is onto, and for two homomorphism $%
f,g:G\rightarrow H,$ there is an automorphism $\alpha \in \mathrm{Aut}(H)$
such that $f=\alpha g.$ For example, each core graph $H$ is uniquely $H$-colourable.  Note that a uniquely $K_{2}$-colourable
graph is precisely a connected bipartite graph.

\begin{lemma}\label{lemma}
\begin{enumerate}
\item For a fixed homomorphism $f:G\rightarrow H$, the preserving
automorphisms relative to $f$ form a subgroup of $\mathrm{Aut}(G).$

\item Distinguishing homomorphisms do not compose, in general.

\item If $f:G\rightarrow H$ is a homomorphism and $\beta \in \mathrm{Aut}(H),
$ then $f$ is distinguishing homomorphism if and only if $\beta f$ is
distinguishing homomorphism.

\item If $G$ is uniquely $H$-colourable, then either all or no homomorphisms
$f:G\rightarrow H$ are distinguishing.

\item Let $G_1$ and $G_2$ be connected, non-isomorphic graphs with disjoint vertex sets. If $f_1:G_1\rightarrow H$ and $f_2:G_2 \rightarrow H$
are distinguishing homomorphisms, then so is $f_1 \cup f_2: G_1 \cup G_2 \rightarrow H.$
\end{enumerate}
\end{lemma}

\begin{proof}
For item (1), suppose that $\alpha _{1}$ and $\alpha _{2}$ are preserving
automorphisms of $G.$ Then we have that%
\begin{equation*}
\alpha _{1}\alpha _{2}f^{-1}=\alpha _{1}f^{-1}=f^{-1}.
\end{equation*}
It is clear that the identity $1$ is a preserving automorphism relative to $%
f.$ Further, note that $\alpha _{1}f^{-1}=f^{-1}$ implies that $\alpha
_{1}^{-1}f^{-1}=f^{-1},$ and so item (1) follows.

For (2), consider the graphs and homomorphisms displayed in Figure~2.
\begin{figure} [h!]
\begin{center}
\epsfig{figure=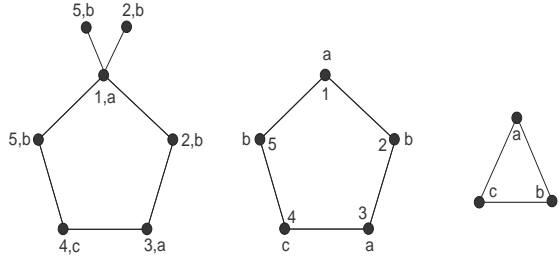, width=3in, height=1.5in}
\caption{Distinguishing homomorphisms which do not compose.}\label{fig2}
\end{center}
\end{figure}
The notation $i,j$ on vertices of the leftmost graph denotes two homomorphisms: the first number $i$ is a homomorphism from the leftmost graph to $C_5$ (which is distinguishing),
and the second letter $j$ is the composed
homomorphism to $K_3$. The reader can verify that the composition of these two distinguishing homomorphisms is not distinguishing.

For (3), suppose that $f$ is distinguishing (the reverse direction is
similar and so is omitted). Fix $\alpha \in \mathrm{Aut}(G).$ Suppose that
\begin{equation*}
\alpha \left( \beta f\right) ^{-1}=\left( \beta f\right) ^{-1}.
\end{equation*}
Then $\alpha f^{-1}\beta ^{-1}=f^{-1}\beta ^{-1}.$ Fix $x\in V(H).$ Then
there is a $y\in V(H)$ such that $\beta ^{-1}y=x.$ Hence, $\alpha
f^{-1}\beta ^{-1}(y)=f^{-1}\beta ^{-1}(y)$ implies that $\alpha
f^{-1}(x)=f^{-1}(x).$ As $x$ was arbitrary we have that $\alpha$ is preserving relative to $f,$ and so $\alpha =1$.

Item (4) follows immediately from (3). For item (5), suppose that
\begin{equation}\label{oo}
\alpha (f_1\cup f_2)^{-1}=(f_1\cup f_2)^{-1},
\end{equation}
for
$\alpha \in \mathrm{Aut}(G_1 \cup G_2).$ As $G_1$ and $G_2$ are
not isomorphic, connected, and have disjoint vertex sets, we must have that $\alpha_i=\alpha \upharpoonright G_i$ are automorphisms of $G_i,$ for $i=1,2.$ By (\ref{oo}), we have that $\alpha_if_i^{-1} = f_i^{-1},$ which implies for $i=1,2$ that $\alpha_i=1,$ and so $\alpha=1.$
\end{proof}

\section{Main results}\label{mainsec}

A graph satisfies the \emph{connected existentially closed} or \emph{c.e.c}.\
adjacency property if for all non-joined vertices $u$ and $v$ (which may be equal) and
finite sets of vertices $T$ not containing $u$ or $v,$ there is a path $P$
of length at least $2$ connecting $u$ and $v$ with the property that no
vertex of $P\backslash \{u,v\}$ is joined to a vertex of $T.$ (Note that if $%
u=v,$ then $P$ is a closed path connected to $u$ with at least one vertex not
equalling $u.)$ See Figure~\ref{fig3}.
\begin{figure} [h!]
\begin{center}
\epsfig{figure=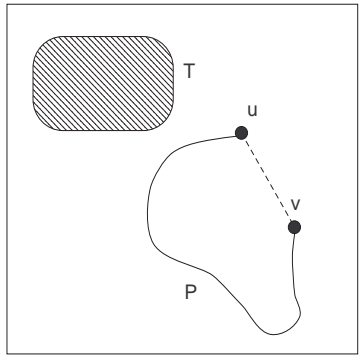,width=2in,height=2in}
\caption{The c.e.c.\ property.}\label{fig3}
\end{center}
\end{figure}
In particular, the \emph{internal vertices} of $P$ are distinct from and
not joined to a vertex of $T.$

The infinite random graph is c.e.c.\ as it is e.c. The infinite random bipartite graph is also c.e.c. To see this, note
that if $u$ and $v$ are the same colour, then they have infinitely many
common neighbours, and so they must have one outside $T$. This gives a path of
length $2$ connecting $u$ and $v$ with the desired properties. If $u$ and $v$
are different colours, then consider a neighbour $w_{1}$ of $u$ distinct
from $v$ and any element of $T.$ We may find a common neighbour $w_{2}$ of $v
$ and $w_1$ not equalling $u$ or a vertex of $T.$ Then the path $P=uw_{1}w_{2}v
$ has the desired properties.

We now state our main result. Given graphs $X$ and $Y,$ define their \emph{%
join}, written $X\vee Y,$ by adding all edges between disjoint copies of $X$
and $Y.$

\begin{theorem}
\label{main}If $G\rightarrow H$ and $G\ $is c.e.c., then there are $2^{\aleph_0}$ distinct
distinguishing homomorphisms from $G$ to $H\vee K_{2}.$
\end{theorem}
We defer the proof of Theorem~\ref{main} to Section~\ref{proofs}, and first focus on applications to certain infinite graphs with bounded chromatic number.

Let $H$ be a finite, non-trivial, connected graph.
As studied in \cite{bonato0} and later in \cite{mms}, there is a certain
class of countable universal graphs admitting a homomorphism into $H$; these are
defined in terms of uniquely $H$-colourable graphs. For each core graph $H,$
there is a uniquely $H$-colourable graph $M(H)$ which is unique up to isomorphism with the
following properties.

\begin{enumerate}
\item[(M1)] Each finite $H$-colourable graph is isomorphic to an induced
subgraph of $M(H).$

\item[(M2)] Each finite induced subgraph $X$ of $M(H)$ is contained in a finite
uniquely $H$-colourable subgraph $X^{\prime }$ of $M(H).$

\item[(M3)] If $X$ is a uniquely $H$-colourable induced subgraph of $M(H),$ and $X$
is an induced subgraph of a uniquely $H$-colourable graph $Y$, then there
is an isomorphic copy $Y^{\prime }$ of $Y$ in $M(H)$ and an isomorphism $%
\alpha :Y\rightarrow Y^{\prime }$ such that $\alpha \upharpoonright X=1_{X}.$
\end{enumerate}

Property (M3) is sometimes referred to as \emph{amalgamating }$Y$\emph{\ into
}$M(H)$\emph{\ over} $X,$ and it can be viewed as a certain kind of
adjacency property for $M(H).$ The graph $M(H)$ is sometimes called \emph{%
universal} \emph{pseudo-homogeneous }(since every isomorphism of finite uniquely $H$-colourable
induced subgraphs of $M(H)$ extends to an automorphism; for more on such graphs see Chapter~11
of Fra\"{\i}ss\'{e} \cite{fra}).

We note that each $H$-colourable graph is an induced subgraph of a uniquely $%
H$-colourable graph via the following construction. Assume $G$ and $H$ are
disjoint. Fix a homomorphism $f:G\rightarrow H$ and define $G(f)$ to be the
graph with vertices $V(G)\cup V(H)$ and edges:
\begin{equation*}
E(G)\cup E(H)\cup \{xy:x\in V(G),y\in V(H),f(x)y\in E(H)\}.
\end{equation*}
The graph $G(f)$ is the \emph{fixation} of $G$ by $f$ relative to $H;$ see
Figure~\ref{fig5}.
\begin{figure}[h!]
\begin{center}
\epsfig{figure=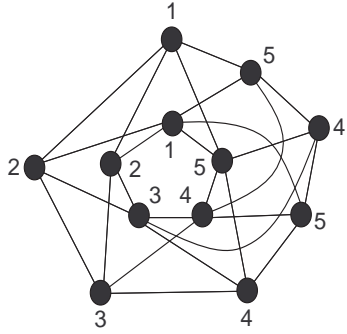,width=2in,height=2in}
\caption{A fixation of $C_7$ by $C_5.$ The $C_5$-colouring of
$C_7$ is shown as the labelling of the vertices of $C_7.$}\label{fig5}
\end{center}
\end{figure}
We restate the following result from \cite{bonato1}.

\begin{theorem}[\cite{bonato1}]\label{fixa}
Suppose that $H$ is a core graph, and if $f:G\rightarrow H$ is a homomorphism, then $G(f)$ is
uniquely $H$-colourable, and $f\cup 1_{G}:G(f)\rightarrow H$ is a
homomorphism.
\end{theorem}

From Theorem~\ref{main} we have the following result.

\begin{corollary}
For all non-trivial, connected graphs $H,$ $M(H)\overset{D}{\rightarrow }H\vee K_{2}.$ In particular,
$\chi_D(M(H))\le \chi(M(H)) +2.$
\end{corollary}

\begin{proof}
As $M(H)\rightarrow H,$ it is sufficient to show that $M(H)$ is c.e.c. Fix
non-joined vertices $u$ and $v$ and a finite set of vertices $T$ in $M(H)$
not containing $u$ or $v.$ Let $X$ be the subgraph of $M(H)$ induced by $\{u,v\}\cup T;
$ by (M2), there is a finite uniquely $H$-colourable graph $X^{\prime }$ in $%
M(H)\ $containing $X.$ Fix a homomorphism $f:X^{\prime }\rightarrow H.$

Suppose that $f(u)=f(v).$ As $H$ is connected and non-trivial, there is a
vertex $i$ of $H$ joined to $f(u).$ We then add a new vertex $z$ to $%
X^{\prime }$ joined to $u$ and $v,$ to form the path $Q.$  The resulting
graph $X^{\prime \prime }$ is $H$-colourable by mapping $X^{\prime }$ via $f$
and sending $z$ to $i.$

If $f(u)\neq f(v),$ then fix a path $Q^{\prime }$ connecting $f(u)$ and $f(v)
$ in $H.$ We may add a path $Q$ (the same length as $Q^{\prime }$ and so
that no internal vertex is joined to a vertex of $X^{\prime })$ to $%
X^{\prime }$ connecting $f(u)$ and $f(v),$ so that each vertex of $Q$ is
mapped to the corresponding vertex of $Q'.$ Let $X^{\prime }$, along with the path $%
Q$ form the graph $X^{\prime \prime }.$

In either case, the resulting graph $X^{\prime \prime }$ contains $X^{\prime
}$ as an induced subgraph and admits a homomorphism, say $f^{\prime \prime },
$ to $H.$ Now form the fixation $X^{\prime \prime }(f^{\prime \prime })=Y.$
By Theorem~\ref{fixa}, $Y$ is uniquely $H$-colourable, and so by (M3) we may
find an induced subgraph $Y^{\prime }$ of $M(H)$ and an isomorphism $\alpha
:Y\rightarrow Y^{\prime }$ such that $\alpha \upharpoonright X^{\prime
}=1_{X^{\prime }}.$ In particular, $\alpha(Q)$ is a path connecting $u$ and $v$
whose internal vertices are disjoint from the set $T.$
\end{proof}

An open problem is whether $M(G)\overset{D}{\rightarrow }H\vee
K_{1}.$ In the case $M(K_2)$, which is isomorphic to the infinite random bipartite graph, this would imply that $\chi _{D}(M(K_2))=3$ (it is not $2$, since by Theorem 2.4 of \cite{cht} a connected graph $G$ with $\chi _{D}(G)=2$ has an automorphism group that has order $1$ or $2$).

\section{Proof of Theorem~\ref{main}}\label{proofs}

Consider the tree $T_{\infty}$ in Figure~\ref{rayless} formed by adding a path
of each finite length to the root vertex of infinite degree.
\begin{figure} [h]
\begin{center}
\epsfig{figure=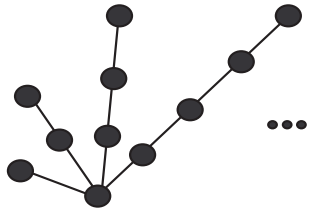, width=2in, height=1.5in} \caption{The tree $T_{\infty}.$}
\label{rayless}
\end{center}
\end{figure}
Label the \emph{branch} (that is, a path connected to the root) of this tree with length $i$ by $b_i.$ Let $\mathcal{Z}$ be the set of infinite-co-infinite subsets of the positive integers. Note that $|\mathcal{Z}|=2^{\aleph _0}.$ For $S\in \mathcal{Z}$, form the sequence $s$ listing the elements of $S$ in increasing order. Note that $s$ is unbounded. We define a tree $T_s$ to be the induced subgraph of $T_{\infty}$ by deleting each branch $b_i$ where $i$ is not listed in $s.$ Note that each $T_s$ has a trivial automorphism group.

We first prove the following lemma.

\begin{lemma}\label{lemma2}
Fix $s \in \mathcal{Z}$.
If $G$ is c.e.c., then there is a partition $A,B$ of $V(G)$ such that the subgraph induced by $B$ is isomorphic to $T_s,$ and for all distinct vertices $x$ and $y$ in $A,$ there is a $z\in B$ such that $z$ is joined to exactly one of $x$ or $y.$
\end{lemma}

\begin{proof} Let $G^{[2]}$ be the set of all unordered pairs of vertices from $G$,  We will define sets of vertices $B_t$ such that $B_t \subseteq B_{t+1}$ for all $t\ge 1.$ Each pair in $G^{[2]}$ will be exactly one of \emph{processed} or \emph{unprocessed}, and exactly one of \emph{good} or \emph{bad}. We proceed over an infinite sequence of time-steps to process pairs. In a given time-step $t$, let $\mathrm{PROC}(t)$ be the set of processed pairs, and $\mathrm{GOOD}(t)$ be the set of good pairs. We set $\mathrm{GOOD}(0)=G^{[2]}$, and let $\mathrm{PROC}(0)$ and $B_0$ be empty. Order the pairs in $G^{[2]}$ as $(\{x_i,y_i\}:i\in \mathbb{N}^+)$. The idea of the proof is to process all pairs so that vertices in the processed good pairs form the set $A,$ and the vertices of $B$ are chosen from vertices in bad pairs. Further, we ensure that for processed good pairs $\{ x,y\}$ there is a $z\in B$ such that $z$ is joined to exactly one of $x$ or $y.$ The subgraph induced by $B$ will be isomorphic to $T_s.$

By the c.e.c.\ property with $u=v=x_1$ and $T=\{y_1\}$, there is a vertex $z_1$ joined to $x_1$ and neither joined nor equal to $y_1.$ Let $B_1 =\{z_1\}.$ The vertex $z_1$ will play the role of the root in $T_s$. The pair $\{ x_1,y_1\}$ is now processed. A pair in $G^{[2]}$ containing $z_1$ is bad and processed; all remaining pairs form $\mathrm{GOOD}(1)$. Let $\mathrm{PROC}(1)$ be the set of processed pairs so far, and note that $\mathrm{PROC}(1)\cap \mathrm{GOOD}(1)$ contains the single element $\{ x_1,y_1\}$.

For some $t\ge 0$ assume that $\mathrm{GOOD}(t)$, $\mathrm{PROC}(t)$ and $B_t$ are defined with the following properties.
\begin{enumerate}
\item $\{ \{x_i,y_i\}:1\le i\le t\} \subseteq \mathrm{PROC}(t)$, and $\mathrm{PROC}(t)\cap \mathrm{GOOD}(t) \subseteq \{ \{x_i,y_i\}:1\le i\le t\}$.
\item If $\{x_i,y_i\}\in \mathrm{PROC}(t)\cap \mathrm{GOOD}(t),$ then there is a $z \in B_t$ joined to exactly one of $x_i$ or $y_i.$
\item The subgraph induced by $B_t$ is finite, and contains the first $t$ branches of $T_s$ (and possibly other branches).
\item A pair containing a vertex in $B_t$ is bad; all other pairs are in $\mathrm{GOOD}(t).$
\item Vertices in $B_t$ are not equal to any vertex in a pair in $\mathrm{PROC}(t)\cap \mathrm{GOOD}(t)$.
\end{enumerate}

We now let $\{x_i,y_i\}$ be the first good pair in $G^{[2]}\setminus \mathrm{PROC}(t)$. Note that $i\ge t+1$ by property (1), and such a pair exists by (3) and (4). We will add to $B_t$ the shortest branch of $T_s$
that does not already appear there; without loss of generality, say it is branch $b_k$, with $k\ge t+1$ by (3). To accomplish this, let $T'$ be the vertices in a pair in $\mathrm{PROC}(t)\cap \mathrm{GOOD}(t)$, along with vertices in $B_t\cup\{x_{i},y_{i}\}$ (note that by (1) and (3), $T'$ is finite). By the c.e.c.\ property applied as when $t=1$, there is a vertex $z^{1}$ joined to $z_1,$ and not joined and not equal to any vertex in $T'$. Iterate this process so there is an induced path $P^k=z^1z^2\cdots z^{k}$ joined to $z_1$, and so vertices of the path are not joined nor equal to a vertex in $T'$. Note that we have now added a new branch of length $k$ in $T_s$ to $B_t$, and vertices in this branch are not joined to any other vertex at time $t$ except $z_1.$ We refer to this construction for brevity as \emph{adding a branch of length $k$ to} $z_1$ (observe that $k$ was arbitrary, so we could add any length branch).

We next process $\{x_i,y_i\}.$ Let $T''$ be the vertices in $P^k$ union $T'$. By the c.e.c.\ property, there is a vertex $z_1'$ joined to $z_1$ and to no vertex in $T''.$ In particular, $z_1'$ is not joined to $x_i$. Let $T^{(3)}$ be $T''$ minus the vertices in $\mathrm{PROC}(t)\cap \mathrm{GOOD}(t)$ equalling one of $x_i$  (which may happen since $\mathrm{PROC}(t)\cap \mathrm{GOOD}(t)$ contains unordered pairs). Let $T=T^{(3)} \cup \{y_i\}.$ By the c.e.c.\ property, there is a path $P$ joining $z_1'$ to $x_i$, whose internal vertices are not joined nor equal to a vertex in $T.$ Note that the vertex $x_i$ is joined to a vertex $z$ in $P$ with $z$ not joined nor equal to $y_i.$

Observe that the path $P'=z_1z_1'P$ may not have the length of a branch in $T_s$, or it may be the length of a branch already added. However, we can add a branch of appropriate length at $z$ to lengthen $P'$ to a path $Q$ which is a branch in $T_s$, so that the branch has length different than $k$ and has length different than any branch in $B_t$. Let $B_{t+1}$ be $B_t$ along with vertices of $P^k$ union $Q.$ Any pair in $G^{[2]}$ containing a vertex from $P^k$ or $Q$ becomes bad and processed; let all remaining pairs form $\mathrm{GOOD}(t+1)$. Note that none of the good pairs in $\mathrm{PROC}(t)\cap\mathrm{GOOD}(t)$ become bad; further, $\{x_i,y_i\}$ remains good. We change the status of $\{x_i,y_i\}$ to processed, and add all newly processed pairs to $\mathrm{PROC}(t)$ to form $\mathrm{PROC}(t+1)$. Note that $\mathrm{PROC}(t+1)$, $\mathrm{GOOD}(t+1)$, and $B_{t+1}$ satisfy items (1)-(5).

As $t$ tends to infinity, every pair becomes processed and exactly one of good or bad. Now let $A$ to be vertices which are in some good pair.
Define $B$ to be the union of all the sets $B_t.$ Then $A$ and $B$ partition $V(G)$, the subgraph induced by $B$ is isomomorphic to $T_s$, and for all distinct vertices $x$ and $y$ in $A,$ there is a $z\in B$ such that $z$ is joined to exactly one of $x$ or $y.$\end{proof}

With Lemma~\ref{lemma2} we may now complete the proof of Theorem~\ref{main}.

\begin{proof}[Proof of Theorem~\ref{main}]
Fix $s \in \mathcal{Z}$, and consider a partition $A$ and $B$ of $V(G)$ as in Lemma~\ref{lemma2} so that the subgraph induced by $B$ is isomorphic
to $T_s$. As $|\mathcal{Z}|=2^{\aleph_0},$ it is sufficient to find distinguishing homomorphisms $g_s$ from $G$ to $H\vee K_{2}$ such that $s \neq s'$ implies $g_s \neq g_{s'}.$ We can accomplish the latter assertion by ensuring that $g_s$ maps $A$ maps to $H$ and $B$ maps to $K_2$ (observe that the preimage of $K_2$ induces a subgraph isomorphic to $T_s$).

Fix $f:G\rightarrow H$ a homomorphism, and label the vertices of $K_{2}$
(that is, the $K_{2}$ outside $H)$ by $1$ and $2.$ Let $f_A$ be the restriction of $f$ on $A.$ Define a homomorphism $%
f_B:B\rightarrow K_{2}$ such each odd distance vertex from the
root of $B$ is labelled $2,$ and the remaining vertices are labelled $1.$
Define
\begin{equation*}
g_s=f_A\cup f_B:G\rightarrow H\vee K_{2}
\end{equation*}
and note that this mapping is a homomorphism. Suppose that some automorphism
of $G,$ say $\alpha $, is preserving relative to $g_s.$ It is easy to see that
$g_s \upharpoonright B$ is the identity on $B$. Suppose that for some distinct
vertices $x$ and $y$ in $A,$ $g_s(x)=y.$ By the properties of $A$ and $B$,
there is a vertex $z$ in $B$ joined to $x$ (say) and not $y.$ But
this contradicts the fact that $g_s$ fixes $z$.
\end{proof}

\end{document}